\documentclass[12pt]{amsart}
\usepackage{amsmath,amssymb}

\numberwithin{equation}{section}
\newtheorem{theorem}[equation]{Theorem}
\newtheorem{proposition}[equation]{Proposition}
\newtheorem{lemma}[equation]{Lemma}

\theoremstyle{remark}

\newtheorem{remark}[equation]{Remark}

\DeclareMathOperator{\Tr}{Tr}
\newcommand{\F}{\mathbb{F}}

\usepackage[colorlinks,pagebackref,pdftex, bookmarks=false]{hyperref}

\begin{document}

\title[Permutation polynomials]{Permutation polynomials of the form $X +\gamma\Tr(X^k)$}

\author{Gohar Kyureghyan}
\address{Otto-von-Guericke University, Magdeburg, Germany}
\email{gohar.kyureghyan@ovgu.de}

\author{Michael E. Zieve} \address{Department of Mathematics,
University of Michigan,
Ann Arbor, MI 48109-1043, USA}
\email{zieve@umich.edu}

\thanks{
The authors thank the referee for providing several corrections and helpful suggestions.
The second author was partially supported by NSF grant DMS-1162181.}

\maketitle

\section{Introduction} \label{kyureg-zieve:seclabel1}

Let $\F_q$ be the finite field of size $q$, and let $p$ be the characteristic of $\F_q$.
A \emph{permutation polynomial} over $\F_q$ is a polynomial $F(X) \in \F_q[X]$ for which the induced function $x\mapsto f(x)$
is a permutation of $\F_q$.
Permutation polynomials are used in various applications of finite fields, where special importance is attached
to permutation polynomials which have few terms.  Due in part to this, and also due to the intrinsic interest of the problem,
the study of permutation polynomials with few terms has thrived for well over a century.  In this paper we make a new
contribution to this topic by analyzing permutation polynomials of the form $X+\gamma \Tr_{q^n/q}(X^k)$ where $n$ and $k$
are positive integers, $\gamma\in\F_{q^n}^*$, and $\Tr_{q^n/q}(X):=X^{q^{n-1}}+X^{q^{n-2}}+\dots+X^q+X$ (so that $\Tr_{q^n/q}$
induces the trace map from $\F_{q^n}$ to $\F_q$).  We produce the following families of permutation polynomials:

\begin{theorem}\label{main}
The polynomial $X + \gamma \Tr_{q^n/q}(X^k)$ is a permutation polynomial over\/ $\F_{q^n}$ in each of the following cases:
\begin{description}
\item[(a)] $n=2$, $q\equiv \pm 1\pmod{6}$, $\gamma=-1/3$, $k=2q-1$
\item[(b)] $n=2$, $q\equiv 5\pmod{6}$, $\gamma^3=-1/27$, $k=2q-1$
\item[(c)] $n=2$, $q\equiv 1\pmod{3}$, $\gamma=1$, $k=(q^2+q+1)/3$
\item[(d)] $n=2$, $q\equiv 1\pmod{4}$,
$(2\gamma)^{(q+1)/2}= 1$, $k=(q+1)^2/4$
\item[(e)] $n=2$, $q=Q^2$, $Q>0$ odd, $\gamma=-1$, $k=Q^3-Q+1$
\item[(f)] $n=2$, $q=Q^2$, $Q>0$ odd, $\gamma=-1$, $k=Q^3+Q^2-Q$
\item[(g)] $n=3$, $q$ odd, $\gamma=1$, $k=(q^2+1)/2$
\item[(h)] $n=3$, $q$ odd, $\gamma=-1/2$, $k=q^2-q+1$
\item[(i)] $n=2\ell r$, $\gamma^{q^{2\ell}-1}=-1$, $k=q^{\ell}+1$ for some positive integers $\ell$ and $r$.
\end{description}
\end{theorem}

We remark that ``at random" one would not expect there to be permutation polynomials of the form $X+\gamma\Tr_{q^n/q}(X^k)$
over $\F_{q^n}$ when $Q:=q^n$ is sufficiently large and $\gamma\in\F_{Q}^*$, since the polynomials of this form induce fewer than $Q^2$
distinct functions on $\F_{Q}$ while the proportion of permutations amongst all functions $\F_{Q}\to\F_{Q}$ is $Q!/Q^Q$.
Hence it seems that any infinite sequence of examples such as those in Theorem~\ref{main} should exist for some ``good reason", and
it is an interesting challenge to understand the different sorts of reasons which can account for such unexpected examples to occur.

An intriguing feature of the present paper is that we need several different methods in order to prove the various cases of
Theorem~\ref{main}.  In cases (a)--(d) and (g) we use a well-known result (Proposition~\ref{zlem})
asserting that $X h(X^{q-1})$ permutes $\F_{q^n}$ if and only
if $X h(X)^{q-1}$ permutes the set $\mu_{(q^n-1)/(q-1)}$ of $(q^n-1)/(q-1)$-th roots of unity.  We then reformulate the latter condition in a simpler way (see
Proposition~\ref{red-sq2}), so that cases (a) and (b) boil down to showing that a certain degree-$3$ rational function permutes
$\mu_{q+1}$, and case (c) boils down to showing that a certain power map $X^N$ permutes $\mu_{q+1}$.  Case (d) is more interesting,
since we wind up showing that $Xh(X)^{q-1}$ permutes $\mu_{q+1}$ by showing that it acts as $c_1^2 X$ on the nonsquares in $\mu_{q+1}$
and $c_2^2 X^N$ on the squares, for certain elements $c_1,c_2\in\mu_{q+1}$; in this case it seems complicated to analyze $X h(X)^{q-1}$
at once on all elements of $\mu_{q+1}$, rather than treating the squares and nonsquares separately.
In case (g) we use a double application of Proposition~\ref{zlem}, by first saying that the given polynomial permutes $\F_{q^3}$
if and only if a certain associated polynomial permutes $\mu_{q^2+q+1}$, then composing the associated polynomial with certain
power maps $X^N$ which permute $\mu_{q^2+q+1}$, and then reducing the composition mod $X^{q^2+q+1}-1$ to obtain a polynomial
which (again by Proposition~\ref{zlem}) permutes $\mu_{q^2+q+1}$ if and only if a certain additive polynomial permutes $\F_{q^3}$;
we then finish the proof by directly verifying this last bijectivity.
This kind of approach does not seem to work in cases (e) and (f), so for these cases we use a variant of Dobbertin's method~\cite{D}, which
involves computing Gr\"obner bases of several ideals in a multivariate polynomial ring.  
This variant involves new features compared to previous applications of Dobbertin's approach, and may be of independent
interest.
We prove case (h) by composing with certain bijective power maps and additive polynomials to obtain a function which permutes
both the squares and the nonsquares in $\F_{q^3}$; this method was inspired by the notion of affine equivalence of planar functions.
Finally, in case (i) we use an additive analogue of Proposition~\ref{zlem} (see Proposition~\ref{prop-surj})
in order to show that the stated polynomial permutes $\F_{q^n}$ if and only if certain associated polynomials permute $\F_q$, and
we calculate that each of these associated polynomials induces the same function on $\F_q$ as does $X+c$ for some $c\in\F_q$.

We computed all permutation polynomials over fields $\F_{q^n}$ of the form $X+\gamma\Tr_{q^n/q}(X^k)$ with $\gamma\in\F_{q^n}^*$,
$q$ odd, $n>1$, and $q^n<5000$.  Here the hypothesis $n>1$ is needed to distinguish our study from that of permutation binomials.
It is easy to see that, for fixed $q,n,\gamma$, we can replace $k$ by $qk$ or by $k+q^n-1$ without affecting whether $X+\gamma\Tr_{q^n/q}(X^k)$
permutes $\F_{q^n}$.  Modulo this equivalence, the only permutation polynomials arising in our computation which are not listed in Theorem~\ref{main}
are:
\begin{itemize}
\item Cases with $k=p^i$, where the polynomial is automatically a member of the much-studied class of \emph{additive} permutation polynomials.
\item Cases with $p\mid n$ and $(q^n-1)\mid k(q^{n/p}-1)$, where the polynomial induces the identity map on $\F_{q^n}$.
\item $q=7$, $n=2$, $k=10$, $\gamma^4=1$
\item $q=9$, $n=2$, $k=33$, $\gamma^2-\gamma=1$
\item $q=27$, $n=2$, $k=261$, $(\gamma-1)^{13}=\gamma^{13}$
\item $q=9$, $n=3$, $k\in\{11,19,33,57\}$, $\gamma^4=-1$
\item $q=49$, $n=2$, $k=385$, $\gamma^5=-1$.
\end{itemize}
Thus, it seems that Theorem~\ref{main} may explain the bulk of all permutation polynomials of this form, so long as we exclude the simple cases where
the polynomial is additive or induces the identity map.

Permutation polynomials of the form $X+\gamma\Tr_{q^n/q}(X^k)$ were studied in \cite{charpin-kyureg-seta, charpin-kyureg-fq}.
All examples in those papers had the special feature that $\gamma$
is a linear translator of $f(X):=\Tr_{q^n/q}(X^k)$, in the sense that there is some $\delta\in\F_q$ for which
\[
f(x+u\gamma)-f(x)=u\delta
\]
for all $x\in\F_{q^n}$ and $u\in\F_q$.
This property is satisfied (with $\delta=0$) for the polynomials in case (i) of Theorem~\ref{main}, as we show in Remark~\ref{finalrem}.
However, the permutation polynomials in cases (a)--(h) of Theorem~\ref{main} are of a different nature, since for these polynomials
$\gamma$ is \emph{not} a linear translator of $\Tr_{q^n/q}(X^k)$.

This paper is organized as follows.  In Section~\ref{kyureg-zieve:general} we prove some general properties about maps of
the form $X+\gamma\Tr_{q^n/q}(X^k)$.
Then in Sections \ref{kyureg-zieve:case1}--\ref{kyureg-zieve:case9} we prove Theorem~\ref{main} in each of cases (a)--(i).
In addition, we note that Theorems \ref{theorem1}, \ref{case5}, \ref{case6} and \ref{case8} list further  families of sparse rational functions which permute
either a finite field or a group of roots of unity in a finite field.

\section{General remarks and equivalent statements}\label{kyureg-zieve:general}

We begin with two propositions which reformulate the condition that $X+\gamma\Tr_{q^n/q}(X^k)$ should permute 
$\F_{q^n}$ in terms of properties
of some associated functions on $\F_q$.  These reformulations play a crucial role in our proof of Theorem~\ref{main}.
In fact the reformulations apply to a more general class of functions.

\begin{proposition}\label{prop-surj}
For any function $f\colon\F_{q^n}\to\F_q$ with $n\ge 2$, and any $\gamma\in\F_{q^n}^*$, the following
three statements are equivalent:
\begin{description}
\item[(a)]
The map
$F\colon x\mapsto x + \gamma \cdot f(x)$
is bijective on\/ $\mathbb{F}_{q^n}$.
\item[(b)] For each $\alpha \in \mathbb{F}_{q^n}$ the map
$x\mapsto x+ f(\alpha +\gamma \cdot x) $ is bijective on\/ $\F_q$.
\item[(c)]
 For each $\alpha \in \mathbb{F}_{q^n}$ there is a unique $x\in\F_q$
 for which $x+ f(\alpha +\gamma \cdot x)=0$.
\end{description}
\end{proposition}

\begin{proof}
For each $\alpha\in\F_{q^n}$, the function $F$ maps the line $\alpha+\gamma\,\F_q$ into itself,
so that $F$ permutes $\F_{q^n}$ if and only if $F$ induces a permutation on each such line.
Explicitly, for $u\in\F_q$ we have 
\[
F(\alpha + \gamma\,u) = \alpha + \gamma(u + f(\alpha +\gamma\,u)), 
\]
so that $F$ permutes the line $\alpha+\gamma\,\F_q$ if and only if the function $u\mapsto u+f(\alpha+\gamma\,u)$
permutes $\F_q$.  Thus (a) and (b) are equivalent.  Since (b) immediately implies (c), it remains only to show
that (c) implies (b), or equivalently that if (b) did not hold then (c) would not hold.  So suppose there is
some $\alpha'\in\F_{q^n}$ for which the function $x\mapsto x+f(\alpha'+\gamma\cdot x)$ is not bijective on $\F_q$.
Then there are distinct elements $u_1,u_2\in\F_q$ which have the same image $y$ under this function.
Hence for $i=1,2$ we have
\[
y = u_i + f(\alpha' +\gamma \cdot u_i),
\]
or equivalently
\[
(u_i - y) +  f(\alpha'+\gamma \cdot (u_i-y)+ \gamma y) = 0.
\]
Thus $x_i:=u_i-y$ satisfies $x_i+f(\alpha'+\gamma y + \gamma x_i)=0$;
since $x_1$ and $x_2$ are distinct elements of $\F_q$, this
contradicts (c) for the value $\alpha:=\alpha'+y\gamma$.
\end{proof}

The next result is a special case of \cite[Lemma~2.1]{Z}.

\begin{proposition}\label{zlem}
For any $h(X)\in\F_{q^n}[X]$, the polynomial $X h(X^{q-1})$ permutes\/ $\F_{q^n}$
if and only if $X h(X)^{q-1}$ permutes the set of $(q^n-1)/(q-1)$-th roots of unity in\/ $\F_{q^n}^*$.
\end{proposition}



We now give further reformulations in case $n=2$.

\begin{proposition}\label{red-sq1}
Let $\gamma, \omega\in\F_{q^2}$ be linearly independent over\/ $\F_q$,
 and let $f\colon \F_{q^2} \to \F_q$ be a function satisfying $f(u\cdot x) = u\cdot f(x)$ for each $u\in \F_q$ and $x \in \F_{q^2}$.
Then
$x\mapsto x +  \gamma\cdot f(x)$ permutes\/ $\F_{q^2}$ if and only if
$f(\gamma) \ne -1$ and 
$x\mapsto x +  f(\omega + \gamma x)$ permutes\/ $\F_q$. 
\end{proposition}
\begin{proof}
By Proposition~\ref{prop-surj}, $x\mapsto x+\gamma\cdot f(x)$ permutes $\F_{q^2}$ if and only if
$F_{\alpha}\colon x\mapsto x+f(\alpha+\gamma\cdot x)$ permutes $\F_q$ for each $\alpha\in\F_{q^2}$.
The elements $\alpha\in\F_{q^2}$ are precisely the elements $\alpha:=u\cdot\gamma+v \cdot\omega$ with $u,v\in\F_q$,
so we treat each choice of $u,v$ in turn.
If $v=0$ then for $x\in\F_q$
we have
\[
F_{\alpha}(x) = x +  f((u+x)\gamma) = x+ (u+x)f(\gamma) = \left(1+ f(\gamma)\right)x + uf(\gamma),
\]
so that $F_{\alpha}$ permutes $\F_q$ if and only if $f(\gamma)\ne -1$.  If $v\ne 0$ then for $x\in\F_q$ we have
\[
F_{\alpha}(x) = x +  f(v\omega + (u+x)\gamma) = v\left(\frac{x}{v} + f\left(\omega + \frac{u+x}{v} \gamma \right)\right),
\]
so that $F_{\alpha}$ permutes $\F_q$ if and only if $x\mapsto x+f(\omega+(x+\frac{u}{v})\gamma)$ permutes
$\F_q$, or equivalently $x\mapsto x+f(\omega+x\gamma)$ permutes $\F_q$.
\end{proof}

\begin{proposition}\label{red-sq2}
Let $k:=(q-1)N+1$ for some integer $N\ge 0$, and pick any $\gamma\in\F_{q^2}^*$.  
Then the polynomial $F(X):=X + \gamma\Tr_{q^2/q}(X^k)$ permutes\/ $\F_{q^2}$
if and only if 
\[
H(X):=\frac{X^N+\gamma^q(1+X^{2N-1})}{X^{N-1}+\gamma(X^{2N-1}+1)}
\]
permutes the set $\mu_{q+1}$ of $(q+1)$-th roots of unity in\/ $\F_{q^2}^*$, or equivalently
this rational function is injective on $\mu_{q+1}$ and its denominator has no roots in $\mu_{q+1}$.
\end{proposition}

\begin{proof}
By Proposition~\ref{zlem}, $F(X)$ permutes $\F_{q^2}$ if and only if
\[
G(X):=X\bigl(1+\gamma (X^N+X^{qN+1})\bigr)^{q-1}
\]
permutes $\mu_{q+1}$.
For any $x\in\mu_{q+1}$ such that $G(x)\ne 0$, we have
\begin{align*}
\frac{G(x)}x &= \bigl(1+\gamma(x^N+x^{qN+1})\bigr)^{q-1} \\
&= \frac{1+\gamma^q(x^{Nq}+x^{q^2N+q})}{1+\gamma(x^N+x^{qN+1})} \\
&= \frac{1+\gamma^q(x^{-N}+x^{N-1})}{1+\gamma(x^N+x^{-N+1})} \\
&=\frac{H(x)}x.
\end{align*}
Therefore $G(X)$ permutes $\mu_{q+1}$ if and only if
$H(X)$ permutes $\mu_{q+1}$.
Finally, if the denominator of $H(X)$ has no roots in $\mu_{q+1}$
then the above computation shows that $G$ and $H$ agree on $\mu_{q+1}$,
and thus $H(\mu_{q+1})=G(\mu_{q+1})\subseteq\mu_{q+1}$, so that bijectivity
of $H$ on $\mu_{q+1}$ follows from injectivity.
\end{proof}


\section{The case that $n=2$ and $k=2q-1$}\label{kyureg-zieve:case1}

In this section we prove cases (a) and (b) of Theorem~\ref{main}.  We begin with a proof of (a), which also produces
some degree-3 rational functions which permute $\F_q$, as well as some degree-3 rational functions which permute
the set $\mu_{q+1}$ of $(q+1)$-th roots of unity in $\F_{q^2}^*$.
\begin{theorem} \label{theorem1}
If $\gcd(q,6) =1$ then the following are true:
\begin{description}
\item[(a)] $\displaystyle{F_1(X) := X - \frac{1}{3} \Tr_{q^2/q}(X^{2q-1})}$ permutes\/ $\F_{q^2}$. 
\item[(b)] For any nonsquare $\nu \in \F_q$, the function $\displaystyle{\frac{X(X^2 -9\nu)}{X^2 -\nu} }$ permutes\/ $\F_{q}$.  
\item[(c)] $g(X):=\displaystyle{\frac{ X^3-3X^2+ 1}{X^3 -3X +1} }$ permutes $\mu_{q+1}$.
\end{description}
\end{theorem}
\begin{proof}
First we show that (a) and (b) are equivalent. Pick $\omega \in \F_{q^2}$ with $\omega^q = -\omega \ne 0$.
Then $\omega$ and $-1/3$ are linearly independent over $\F_q$.  We apply Proposition~\ref{red-sq1} with
$\gamma=-\frac13$ and $f(X):=\Tr_{q^2/q}(X^{2q-1})$, noting that $f(a\cdot x)=a\cdot f(x)$ for $a\in\F_q$ and $x \in \F_{q^2}$ since $2q-1\equiv 1\pmod{q-1}$.
Since $f(-\frac{1}{3})=-\frac23\ne -1$, it follows that (a) holds if and only if
\[
h(X) := X + \Tr_{q^2/q}\bigl((\omega -\frac{1}{3} X)^{2q-1}\bigr)
\]
permutes $\F_q$. For $x\in\F_q$ we have
\begin{eqnarray*}
h(3x) &=& 3x + \frac{(\omega - x)^{2q}}{\omega -x} + \frac{(\omega - x)^{2q^2}}{(\omega -x)^q}\\
&=& 3x + \frac{(-\omega - x)^{2}}{\omega -x} + \frac{(\omega - x)^{2}}{(-\omega -x)}\\
&=& \frac{x^3-9\omega^2x}{x^2-\omega^2}.
\end{eqnarray*}
Since the above equivalence holds for each $\omega$ with $\omega^q=-\omega\ne 0$, and
the set of all corresponding values $\omega^2$ coincides with the set of nonsquares in $\F_q$,
this shows that (a) and (b) are equivalent.
The equivalence of (a) and (c) follows from Proposition \ref{red-sq2} with $N=2$ and $\gamma=-\frac13$.
So it is enough to verify (c).  By Proposition~\ref{red-sq2}, it suffices to show that 
$g$ is both well-defined and injective on $\mu_{q+1}$.
If some $\alpha\in\mu_{q+1}$ is a root of $r(X):=X^3-3X+1$, then also $\alpha^q$ is a root of $r(X)$; since neither $1$
nor $-1$ is a root of $r(X)$, it follows that $\alpha\ne\alpha^q$.  Since the product of the roots of $r(X)$ is $-1$, the
third root of $r(X)$ must be $-1/\alpha^{q+1}=-1$, which is not the case since $r(-1)\ne 0$.
Thus $g$ is well-defined on $\mu_{q+1}$.
The numerator of $g(X)-g(Y)$ is
\begin{align*}
(X^3-3X^2&+1)(Y^3-3Y+1)-
(Y^3-3Y^2+1)(X^3-3X+1) \\
&=3(X-Y) (XY-X+1) (XY-Y+1).
\end{align*}
Hence if $g(\alpha)=g(\beta)$ for some distinct $\alpha,\beta\in\mu_{q+1}$
then $\alpha\beta\in\{\alpha-1,\beta-1\}$.  Assume without loss that $\alpha\beta=\alpha-1$,
so that $(\alpha-1)^{q+1}=1$.
But
\[
(\alpha-1)^{q+1}=\alpha^{q+1}-\alpha^q-\alpha+1=2-\frac{1}{\alpha}-\alpha,\]
so that $\alpha+\frac{1}{\alpha}=1$ and thus $\beta=1-\frac{1}{\alpha}=\alpha$, a contradiction.
\end{proof}

We now show that the permutation property of the polynomials in case (b) of Theorem~\ref{main}
follows at once from the analogous property in case (a).

\begin{theorem}\label{theorem2}
If $q \equiv 5 \pmod 6$ and $\gamma \in \F_{q^2}$ satisfies $\gamma^3 = -\frac{1}{27}$, then 
\[
F_2(X):= X + \gamma\,\Tr_{q^2/q}(X^{2q-1})
\]
permutes\/ $\F_{q^2}$.
\end{theorem}

\begin{proof}
Since $\omega:=-3\gamma$ satisfies $\omega^3 =1$, we have $\omega^{2q-1} =1$ and thus
\begin{align*}
F_2(\omega X) &=  \omega X - \frac{1}{3} \omega \,\Tr_{q^2/q}(\omega^{2q-1}X^{2q-1}) \\
              & =  \omega \bigl( X - \frac{1}{3}  \,\Tr_{q^2/q}(X^{2q-1}) \bigr),
 \end{align*}
so the result follows from Theorem~\ref{theorem1}.
\end{proof}

\begin{remark}
A different proof of bijectivity of $F_1$ was given in the recent paper \cite{hou}.
\end{remark}


\section{The case that $n=2$ and $k=(q^2+q+1)/3$}
\label{kyureg-zieve:case3}

We now prove case (c) of Theorem~\ref{main}.

\begin{theorem}\label{case3}
If $q \equiv 1 \pmod 3$, then 
\[
F_3(X):= X + \Tr_{q^2/q}(X^{(q^2+q+1)/3})
\]
permutes\/ $\F_{q^2}$.
\end{theorem}
\begin{proof}
By Proposition~\ref{red-sq2} with $N=\frac{q+2}3$, it suffices to show that
\[
g(X):=\frac{X^N+1+X^{2N-1}}{X^{N-1}+X^{2N-1}+1}
\]
permutes  $\mu_{q+1}$.  Here $3N\equiv 1\pmod{q+1}$, so by putting
$Y=X^N$ our condition becomes that
\[\frac{Y+1+\frac{1}Y}{\frac{1}{Y^2}+\frac1Y+1}\] should permute $\mu_{q+1}$.
This function is the identity function $y \mapsto y$ so long as $\mu_{q+1}$ contains no roots of $Y^2+Y+1$,
which is the case since those roots have order $3$.  Hence $g(X)$ induces the
same function on $\mu_{q+1}$ as does $X^N$, so that $g$ permutes $\mu_{q+1}$.
\end{proof}


\section{The case that $n=2$ and $k=(q+1)^2/4$}
\label{kyureg-zieve:case4}

We now prove case (d) of Theorem~\ref{main}.

\begin{theorem}\label{case4}
Let  $q\equiv 1\pmod{4}$, and let $\gamma \in \F_{q^2}$ satisfy
$\left(2\gamma\right)^{(q+1)/2}= 1$. Then
\[
F_4(X):= X +\gamma\, \Tr_{q^2/q}(X^{(q+1)^2/4})
\]
permutes\/ $\F_{q^2}$.
\end{theorem}
\begin{proof}
By Proposition~\ref{zlem}, it suffices to show that
\[
g(X):=X(\gamma^{-1}+X^N+X^{qN+1})^{q-1}
\]
induces a bijection on $\mu_{q+1}$ in case $N=\frac{q+3}4$.
Note that every element of $\mu_{q+1}$ can be written in exactly one way as $\pm y^2$
with $y \in \mu_{(q+1)/2}$.
By hypothesis $2\gamma$ is a square in $\mu_{q+1}$, so that $-2\gamma$ is a nonsquare in $\mu_{q+1}$,
and thus $\gamma^{-1}+2y$ is nonzero for each $y \in \mu_{(q+1)/2}$.
For $y \in \mu_{(q+1)/2}$ we compute
\[
(y^2)^{N}=y^{(q+3)/2}=y \mbox{ and }
(y^2)^{qN+1}=y^{q+2}=y
\]
so that
\begin{eqnarray*}
g(-y^2) &= & -y^2 (\gamma^{-1} + y (-1)^N + y (-1)^{qN+1})^{q-1} \\
  &= & -y^2 \gamma^{1-q} = -(2y\gamma)^2,
\end{eqnarray*}
and
\begin{eqnarray*}
g(y^2) &= & y^2 (\gamma^{-1} + y + y)^{q-1} =
 y^2\, \frac{({\gamma^{-1}}+2y)^q}{{\gamma^{-1}}+2y} \\ 
  &=& y^2\,\frac{4\gamma+2y^{-1}}{\gamma^{-1}+2y} = 2\gamma y.
\end{eqnarray*}
Since $2\gamma$ is in $\mu_{(q+1)/2}$, and squaring is a bijective map on $\mu_{(q+1)/2}$, it follows that
$g$ induces a bijection on $\mu_{(q+1)/2}$ and also a bijection on $-\mu_{(q+1)/2}$, so that $g$ induces
a bijection on $\mu_{q+1}$ as desired.
\end{proof}

\begin{remark}
In fact, the polynomial $F_4$ fixes each nonsquare in $\F_{q^2}$. Moreover, we have
 \begin{eqnarray*}
F_4(x) & = & \left\{ \begin{array}{ll} x +2\gamma \cdot x^{\frac{(q+1)^2}{4}} & \mbox{ if $x$ is a square in } \F_{q^2} \\  x & \mbox{ if $x$ is a nonsquare in } \F_{q^2}.
\end{array} \right .
\end{eqnarray*}
To see this, note that for any $x\in\F_{q^2}$ the element $x^{(q^2-1)/4}$ is either zero or a fourth root of unity, and since $q\equiv 1\pmod{4}$ it follows that $x^{(q^2-1)/4}$ lies in the subfield $\F_q$.  Thus for $x\in\F_{q^2}$ we have
\[
F_4(x) =  x + \gamma \,\Tr_{q^2/q}(x^{\frac{q^2-1}{4}}x^{\frac{q+1}{2}}) = x + \gamma \,x^{\frac{q^2-1}{4}}\,\Tr_{q^2/q}(x^{\frac{q+1}{2}}).
\]
Our claimed expression for $F_4(x)$ now follows from the fact that
\begin{eqnarray*}
\Tr_{q^2/q}(x^{\frac{q+1}{2}}) & = & x^{\frac{q+1}{2}} +  x^{\frac{q^2+q}{2}} \\
& = & x^{\frac{q+1}{2}} (1 +  x^{\frac{q^2-1}{2}}) \\
& = & \left\{ \begin{array}{ll} 2x^{\frac{q+1}{2}} & \mbox{ if $x$ is a square in } \F_{q^2} \\ 0 & \mbox{ if $x$ is a nonsquare in } \F_{q^2}.
\end{array} \right .
\end{eqnarray*}

\end{remark}


\section{The case  that $n=2$ and $q=Q^2$ where $Q$ is an odd prime power and $k=Q^3-Q+1$}
\label{kyureg-zieve:case5}

In this section we prove case (e) of Theorem~\ref{main}, by showing that
\[
 F_5(X) := X - \Tr_{Q^4/Q^2}(X^{Q^3-Q+1})
\]
is a permutation on $\F_{Q^4}$, when $Q$ is an odd prime power.
Our proof relies on a variant of Dobbertin's method~\cite{D}.
We also exhibit certain sparse rational functions which permute either $\F_{Q^2}$ or the set of
$(Q^2+1)$-th roots of unity in $\F_{Q^4}^*$.

\begin{theorem}\label{case5}
For any odd prime power $Q$, we have:
\begin{description}
\item[(a)] $F_5(X) = X - \Tr_{Q^4/Q^2}(X^{Q^3-Q+1})$ 
permutes\/ $\F_{Q^4}$. \vspace*{0.2cm}
\item[(b)]  
$
\displaystyle{ \frac{X^{Q+2} + 3\nu X^Q + 4\nu^{(Q+1)/2}X }{X^{2} - \nu} }
$
permutes\/  $\F_{Q^2}$, where $\nu$ is any nonsquare in\/ $\F_{Q^2}^*$.  \vspace*{0.2cm}
\item[(c)]
$
\displaystyle{\frac{X^{2Q-1} - X^Q +1}{X^{2Q-1} - X^{Q-1}+1}} $ permutes the set of
$(Q^2+1)$-th roots of unity~in~$\F_{Q^4}^*$.
\end{description}
\end{theorem}

We begin with some simple lemmas about the preimages under $F_5$ of some special values.

\begin{lemma}\label{zero}
The only root of $F_5(X)$ in\/ $\F_{Q^4}$ is $0$.
\end{lemma} 
\begin{proof}
Suppose to the contrary that $x\in \F_{Q^4}^*$ is a root of $F_5$.  Then 
$$x^{Q^3-Q} + x^{-Q^3+Q^2+Q-1} = 1,$$
and  $y:=x^{Q^2-1}$ is an element of $\mu_{Q^2+1}$ satisfying  
\begin{equation} \label{eq:m}
 y^Q + y^{1-Q} = 1.
\end{equation}
In particular, $y$ cannot be $\pm 1$.
Since $y^{Q^2}=1/y$, we obtain
$$  y^{-1} + y^{Q+1} = (y^Q + y^{1-Q})^Q = 1^Q = 1^{Q^3} = (y^Q+y^{1-Q})^{Q^3} = y + y^{-Q-1},$$
or equivalently
$$
(y^{Q+2}+1) \cdot (y^{Q}-1) =0.
$$
Since $y \ne 1$, we obtain $y^{Q+2}=-1$, so that $y^{2Q+4}=1$.  Hence the order of $y$ divides 
\[
\gcd(2Q+4,Q^2+1)=2\gcd(Q+2,Q^2+1)=2\gcd(Q+2,5),
\]
so that $y^{10}=1$.
Since $y^{Q+2}=-1$, it follows that $y^5=-1$.
Now (\ref{eq:m}) simplifies to
$$
-y^{-2} - y^3 = 1,
$$
and since $y^5=-1$ this yields the contradiction $0=1$.
\end{proof}

\begin{lemma} \label{claim:3}
The set $S$ of elements $x\in\F_{Q^4}^*$ for which $x^{2Q^2}-x^{Q^2+1}+x^2=0$ is nonempty only
when $3\mid Q$, in which case $S$ consists of the $(Q^2-1)$-th roots of $-1$ and
$F_5$ fixes each element of $S$. 
\end{lemma}

\begin{proof}
Write $u:=x^{Q^2-1}$ with $x\in S$, so that $u^{Q^2+1}=1$ (since plainly $x\ne 0$).
If $3\nmid Q$ then the equation $u^2-u+1=0$ implies that $u$ is a primitive sixth root
of unity, which is impossible since $6\nmid (Q^2+1)$. 
Hence $3\mid Q$, so that $X^{2Q^2}-X^{Q^2+1}+X^2=(X^{Q^2}+X)^2$ and thus $S$ consists
of the $(Q^2-1)$-th roots of $-1$.  Therefore for $x\in S$ and $k:=Q^3-Q+1$ we have
$F_5(x)=x-x^k-x^{Q^2k}=x-x^k-(-x)^k=x$.
\end{proof}

\begin{proof}[Proof of Theorem~\ref{case5}] Write $q:=Q^2$. 
The equivalence of (a), (b) and (c) follows from Propositions~\ref{red-sq1} and \ref{red-sq2} in the same manner as in our previous results.
In the remainder of the proof we show that $F_5$ is bijective on $\F_{Q^4}$.  In light of Lemma~\ref{zero}, it suffices to show that 
each $d\in\F_{Q^4}^*$ has at most one preimage $x$ in $\F_{Q^4}^*$.

Pick any $d$ in $\F_{Q^4}^*$, and write 
\[
e:=d^Q, \,f:=d^{Q^2}, \,g:=d^{Q^3}.
\]
Let $x$ be an element of $\F_{Q^4}^*$ for which $F_5(x)=d$, and write 
\[
y:=x^Q, \,z:=x^{Q^2}, \,w:=x^{Q^3}.
\]
The equations $F_5(x^{Q^i})=d^{Q^i}$ for $i=0,1,2,3$ may be written as
\begin{align}
\label{(1)} x-\frac{xw}{y}-\frac{yz}{w}&=d \\
 \label{(2)} y-\frac{yx}{z}-\frac{zw}x&=e\\
   \label{(3)} z-\frac{zy}w-\frac{wx}y&=f \\
    \label{(4)} w-\frac{wz}x-\frac{xy}z&=g. 
\end{align}

By Lemma~\ref{claim:3}, if the set $S:=\{u\in\F_{Q^4}\colon u^{2Q^2}-u^{Q^2+1}+u^2=0\}$ is nonempty then $3\mid Q$
and $S$ consists of the $(Q^2-1)$-th roots of $-1$, each of which is fixed by $F_5$.
We claim that if $d\in S$ then $x\in S$.
For, if $d\in S$ then $3\mid Q$ and $f=-d$, so that the sum of the left sides of (\ref{(1)}) and (\ref{(3)}) is zero.
The numerator of this sum factors as $(w+y)(wx+zy)$, so we must have either $w=-y$ or $wx=-zy$.
If $w=-y$ then $y^{Q^2-1}=-1$, so that $y\in S$ and hence $x=y^{Q^3}\in S$.
If $wx=-zy$ then $(yz)^{Q^2-1}=-1$, so that $y^{(Q+1)(Q^2-1)}=-1$; but since $y\in\F_{Q^4}^*$ it follows that
the order of $y$ divides $\gcd(2(Q+1)(Q^2-1),Q^4-1)=2(Q^2-1)$, so that $y^{(Q+1)(Q^2-1)}=1\ne -1$, a contradiction.
Thus if $d\in S$ then $x\in S$.

Since $F_5$ fixes every element of $S$, and maps $T:=\F_{Q^4}^*\setminus S$ into itself, it remains to show that
each $d\in T$ has at most one preimage in $T$ under $F_5$.  Thus we assume henceforth that $d,x\in T$, so that
$f^2-df-d^2\ne 0$ and $z^2-xz+x^2\ne 0$, and raising to the $Q$-th power yields
 $g^2-eg-e^2\ne 0$ and $w^2-yw+y^2\ne 0$.
Since $z\ne 0$  we can use (\ref{(2)}) to solve for $w$, obtaining
\begin{equation}\label{(5)}
w = (y-\frac{yx}z-e)\frac{x}z.
\end{equation}
Substituting (\ref{(5)}) into (\ref{(4)}) and simplifying yields
\begin{equation} \label{(6)}
 y = -z\frac{e(x-z)+gz}{x^2+z^2-xz}.
\end{equation}

Substituting (\ref{(5)}) and (\ref{(6)}) into (\ref{(1)}) and (\ref{(3)}) yields $A,B\in\F_{Q^4}(X,Z)$ such that
$A(x,z)=B(x,z)=0$.  Now compute a Gr\"obner basis for the ideal of $\F_{Q^4}[Z',X',V',Z,X]$ generated by the
numerators of $A(X,Z)$ and $B(X,Z)$ as well as the elements $ZZ'-1$, $XX'-1$ and $(X^2-XZ+Z^2)V'-1$.
Since each element of this ideal vanishes when we substitute $\frac{1}z,\frac{1}x,\frac{1}{x^2-xz+z^2},z,x$
for $Z',X',V',Z,X$, we may make this substitution into each element of the Gr\"obner basis to conclude that 
\begin{equation} \label{(7)}
z = x - d + f
\end{equation}
and $C(x)=0$ for a certain degree-$5$ polynomial $C(X)\in\F_{Q^4}[X]$.

Suppose that $x'\in T\setminus\{x\}$ satisfies $F_5(x')=F_5(x)$; then we must have $C(x')=0$.
The coefficients of $C(X)$ are rational functions in $d,e,f,g$; by using (\ref{(1)})--(\ref{(4)})
we may replace these by rational functions in $x,y,z,w$, and after multiplying by a suitable polynomial
in $x,y,z,w$ we can rewrite the resulting polynomial as $(X-x)D(X)E(X)$ where
\begin{align*}
D(X)&:=yw(x^2 - xz + z^2)X^2 + xyz(y-w)(x-z)^2  \\
&\qquad\quad + (z-x)(x^2yw + xy^2z - xyzw - xzw^2 + yz^2w)X  \\
E(X)&:=(x^2y^2-xyzw+z^2w^2)X^2 + wz(y-w)(x-z)^3 \\
&\qquad\quad + (z-x)(x^2yw+xy^2z-xyzw - xzw^2-yz^2w+2z^2w^2)X.
\end{align*}
Hence $x'$ is a root of at least one of $D(X)$ and $E(X)$.
We must have $z\ne x$ (and equivalently $w\ne y$), since otherwise
$D(X)=x^2y^2X^2=E(X)$, which is impossible since $x'$ is a nonzero root of one of these polynomials.
Also $x^2y^2-xyzw+z^2w^2\ne 0$, since otherwise
$(xy)^2-(xy)^{Q^2+1}+(xy)^{2Q^2}=0$ so that $(xy)^{Q^2-1}=-1$, but then the order of $x$
divides both $2(Q+1)(Q^2-1)$ and $Q^4-1$ and hence divides $2(Q^2-1)$, which yields the contradiction
$(xy)^{Q^2-1}=1\ne -1$.
Thus both $D$ and $E$ have degree $2$.

If $D(x')=0$ then $D(x')^Q=0$, so that $(x')^Q$ is a root of the polynomial $D_2(X)$ obtained from $D(X)$ by replacing
each coefficient by its $Q$-th power.  Equations (\ref{(6)}) and (\ref{(7)}) imply that $y=G(x)$ where
\[
G(X):=-(X-d+f)\frac{e(d-f)+g(X-d+f)}{X^2+(X-d+f)^2-X(X-d+f)}.
\]
Since these equations were
deduced from the identity $F_5(x)=d$, it follows that $(x')^Q=G(x')$, so that $x'$ is a root of the numerator of $D_2(G(X))$.
This numerator factors as $E(X)H(X)$ where
\begin{align*}
H(X)&:=
(x^2y^2-xy^2z+xyzw-xzw^2+z^2w^2)X^2 + x^2y(x-z)^2(y-w) \\
&\qquad\quad + (z-x)(2x^2y^2 - x^2yw - xy^2z + xyzw - xzw^2 + yz^2w)X.
\end{align*}
Thus $x'$ is a root of at least one of $E(X)$ and $H(X)$.  If $E(x')=0$ then $x'$ is a common root of $D(X)$ and $E(X)$, so
the resultant of $D(X)$ and $E(X)$ must vanish.  
This resultant equals $-xz(x-z)^4(y-w)^2(x^2y^2-xyzw+z^2w^2)u^{Q+1}$ where
\[
u:=x^2y^2 - x^2yw + xyzw - yz^2w + z^2w^2,
\]
so we must have $u=0$.
But then a routine computation shows that $D(X)=yw(x^2-xz+z^2)(X-x)^2$,
which yields the contradiction $x'=x$.  Thus if $D(x')=0$ then we must have $H(x')=0$.
If $H$ has degree $2$ then as above the resultant of $D(X)$ and $H(X)$ must vanish, so that
\[
0 = xyzw(x-z)^6(y-w)^2(x^2y^2-xyzw+z^2w^2)^{Q+1},
\]
which we know is false.  Thus $\deg H<2$.
Since the constant term of $H(X)$ is nonzero, the condition $H(x')=0$ forces $\deg H=1$, so  the
resultant of $D(X)$ and $H(X)$ is $xy(x-z)^4(w-y)v$ where
\begin{align*}
v&:= x^5y^3w - 2x^4y^4z + x^4y^3zw - 2x^4y^2zw^2 + 3x^3y^4z^2 - 
    3x^3y^3z^2w \\ &\quad + 6x^3y^2z^2w^2 
     - 2x^3yz^2w^3 + x^3z^2w^4 - x^2y^4z^3
    - 2x^2y^3z^3w - x^2y^2z^3w^2 \\ &\quad - x^2z^3w^4 + 2xy^3z^4w - xy^2z^4w^2
    + 2xyz^4w^3 - y^2z^5w^2.
\end{align*}
Thus we must have $v=0$, and also the coefficient of $X^2$ in $H(X)$ must vanish.
But we can express $yz^6w^2(y-w)^3(y^2+w^2)$ as a sum of the products of these two quantities
with certain polynomials in $x,y,z,w$, so that this expression vanishes and thus $y^2=-w^2$, whence
$y^{2Q^2-2}=-1$.  However, there are no $(2Q^2-2)$-th roots of $-1$ in $\F_{Q^4}^*$, since $2Q^2-2$
is divisible by the largest power of $2$ which divides $Q^4-1$.

This completes the proof when $D(x')=0$, and the proof when $E(x')=0$ is similar.
 \end{proof}



\section{The case  that $n=2$ and $q=Q^2$ where $Q$ is an odd prime power and $k=Q^3+Q^2-Q$}
\label{kyureg-zieve:case6}

Case (f) of Theorem~\ref{main} is contained in the following result.

\begin{theorem}\label{case6}
For any odd prime power $Q$, we have:
\begin{description}
\item[(a)]  $F_6(X) := X - \Tr_{Q^4/Q^2}(X^{Q^3+Q^2-Q})$
permutes\/ $\F_{Q^4}$. \vspace*{0.2cm}
\item[(b)] 
$
\displaystyle{ \frac{X^{Q+2} + 3\nu X^Q - 4\nu^{(Q+1)/2}X }{X^{2} - \nu} }
$
permutes\/ $\F_{Q^2}$, where $\nu$ is any nonsquare in\/ $\F_{Q^2}$.  \vspace*{0.2cm}
\item[(c)]
$
\displaystyle{\frac{X^{2Q+1} - X^{Q+1} +1}{X^{2Q+1} - X^{Q}+1}} $ permutes the set of $(Q^2+1)$-th
roots of unity in~$\F_{Q^4}^*$.
\end{description}
\end{theorem}

We note that assertions (b) and (c) of this result are extremely similar to the corresponding assertions
in Theorem~\ref{case5}.  However, we have not been able to find a direct proof that these similar
assertions are logically equivalent to one another.  If we could find such a proof, then Theorem~\ref{case6}
would follow from Theorem~\ref{case5}.  At present, the best we can do is to prove Theorem~\ref{case6} via
a similar argument to our proof of Theorem~\ref{case5}.

\begin{proof}  
The equivalence of (a), (b) and (c) follows from Propositions~\ref{red-sq1} and \ref{red-sq2} in the same manner as in our previous results.
The proof of assertion (a) is nearly identical to the proof of assertion (a) in Theorem~\ref{case5}, so the details
can safely be left to the reader.
\end{proof}


\section{The case that $n=3$ and $k=(q^2+1)/2$}\label{kyureg-zieve:case7}

We now prove case (g) of Theorem~\ref{main}.

\begin{theorem}\label{case7}
If $q$ is odd, then
$$
F_7(X):= X + \Tr_{q^3/q}\big(X^{\frac{q^2+1}{2}}\big)
$$
permutes\/ $\F_{q^3}$.
\end{theorem}
\begin{proof}
By Proposition~\ref{zlem}, it suffices to show that
$$
g(X) := X  (1 + X^{(q+1)/2} + X^{(q^2+q+2)/2} + X^{(q^2+2)(q+1)/2})^{q-1}
$$
permutes the set $\mu_{q^2+q+1}$ of $(q^2+q+1)$-th roots of unity in $\F_{q^3}^*$.
For $x \in \mu_{q^2+q+1}$ we compute
\begin{eqnarray*}
g(x^{2q}) & = & x^{2q}\,(1+x^{-1}+x^{q}+x^{q-1})^{q-1} \\ & = & x^{2q} \cdot (1+x^{-1})^{q-1} \cdot (1+x^{q})^{q-1} \\
 & = & x^{q^2+q} \cdot(1+x^{-1})^{q^2-1} = x^{-1}  (1+x^{-1})^{q^2-1},
\end{eqnarray*}
so that (since $1+x^{-1}\ne 0$)
\[
g(x^{2q})^{-q} = x^q(1+x^{-1})^{-1+q}=x(x+1)^{q-1}.
\]
Again by Proposition~\ref{zlem}, $X(X+1)^{q-1}$ permutes $\mu_{q^2+q+1}$ if and only if $X^q+X$
permutes $\F_{q^3}$, which is indeed the case since $X^q+X$ is an additive polynomial having no nonzero
roots in $\F_{q^3}$.  Therefore $g(X^{2q})^{-q}$ permutes $\mu_{q^2+q+1}$, so that also $g(X)$ permutes $\mu_{q^2+q+1}$,
whence $F_7$ permutes $\F_{q^3}$.
\end{proof}


\section{The case that $n=3$ and $k=q^2-q+1$}\label{kyureg-zieve:case8}

In this section we prove case (h) of Theorem~\ref{main}.

\begin{theorem}\label{case8}
If $q$ is odd then
\[
F_8(X):= X - \frac{1}{2} \Tr_{q^3/q}\big(X^{q^2-q+1}\big)
\]
permutes\/ $\F_{q^3}$.
\end{theorem}

\begin{proof}
Define
\begin{align*}
H(X)&:=2X^{(q^4+q^2)/2} - \Tr_{q^3/q}(X) \\
L(X)&:=2X^q-\Tr_{q^3/q}(X) \\
G(X)&:=L(H(X)).
\end{align*}
First note that $L(X)$ has no nonzero roots in $\F_{q^3}$, since any such root $x$
would satisfy $2x^q=\Tr_{q^3/q}(x)\in\F_q$ so that $x\in\F_q$, whence
$L(x)=-x$.  Since $L(X)$ is an additive polynomial in $\F_{q^3}[X]$, it follows
that $L(X)$ permutes $\F_{q^3}$.  Next, for every $x\in\F_{q^3}$ we have
\[
G(x^2) = L(x)^2.
\]
Writing $S$ for the set of squares in $\F_{q^3}$, it follows that
\[
G(S)=\bigl(L(\F_{q^3})\bigr)^2=S,
\]
since $L$ permutes $\F_{q^3}$.  Let $w$ be any nonsquare in $\F_q^*$, so that
also $w$ is a nonsquare in $\F_{q^3}^*$ and thus $\F_{q^3}=S\cup wS$.
Since all terms of both $L$ and $H$ have degree congruent to $1$~(mod~$q-1$),
the same is true of $G$, so the hypothesis $w^{q-1}=1$ implies that
\[
G(wS)=wG(S)=wS.
\]
Therefore $G$ permutes both $S$ and $wS$, so $G$ permutes $S\cup wS=\F_{q^3}$.
Hence also $H$ permutes $\F_{q^3}$, and since $\gcd(q^2-q+1,q^3-1)=1$ it follows
that $H(X^{q^2-q+1})$ permutes $\F_{q^3}$.  It follows that $F_8(X^q)$ (and hence $F_8(X)$)
permutes $\F_{q^3}$ since $H(0)=0=F_8(0)$ and for $x\in\F_{q^3}^*$ we have
\[
H(x^{q^2-q+1}) = 2x^q - \Tr_{q^3/q}(x^{q^2-q+1}) = 2F_8(x^q). \qedhere
\]
\end{proof}

Our proof
of bijectivity of $H$ is inspired by the proof of \cite[Thm.~3.7]{weng-zeng},  which also distinguished the behavior
of a function on squares and nonsquares.
The key identity  
$L(H(x^2))=L(x)^2 $
in our proof says that  $x\mapsto H(x^2)$  is a planar function which is affine equivalent to
$x\mapsto x^2$. Since the nucleus of any semifield induced by $x\mapsto H(x^2)$ is $\F_{q^3}$, the hypotheses of 
\cite[Thm.~3.7]{weng-zeng}  do not apply in our situation; however, it appears that those hypotheses can be
relaxed to cover both our situation and several others.  Our proof of Theorem~\ref{case8}
uses our key identity to avoid the bulk of the work involved in the proof of \cite[Thm.~3.7]{weng-zeng}.


\section{The case that $k=q^{\ell}+1$ with $\ell\mid n$} \label{kyureg-zieve:case9}

We now prove case (i) of Theorem~\ref{main}.

\begin{theorem}\label{case9}
For any prime power $q$ and any positive integers $\ell,n$ with $2\ell\mid n$, if $\gamma\in\F_{q^n}$ satisfies $\gamma^{q^{2\ell}-1}=-1$
then the polynomial $\displaystyle{F_9(x) := X + \gamma \Tr_{q^n/q}(X^{q^{\ell}+1})}$ permutes\/ $\F_{q^n}$.
\end{theorem}

Note that if $q$ is odd then $\F_{q^n}$ contains elements $\gamma$ as in Theorem~\ref{case9} if and only if $n$ is divisible by $4\ell$.

\begin{proof}
By Proposition~\ref{prop-surj}, $F_9$ permutes $\F_{q^n}$ if and only if for each $\alpha\in\F_{q^n}$
the polynomial $h_{\alpha}(X):=X+\Tr_{q^n/q}\Bigl((\alpha+\gamma X)^{q^\ell+1}\Bigr)$ permutes $\F_q$.
For $x\in\F_q$ and $Q:=q^{\ell}$ we have
\begin{align*}
(\alpha+\gamma x)^{Q+1} &= \gamma^{Q+1} x^{Q+1} + \alpha \gamma^Q x^Q + \alpha^Q \gamma x + \alpha^{Q+1} \\
 &= \gamma^{Q+1} x^2 + \alpha \gamma^Q x + \alpha^Q \gamma x + \alpha^{Q+1},
 \end{align*}
so that
\[
\Tr_{q^n/q}((\alpha+\gamma x)^{Q+1}) =
\Tr_{q^n/q}(\gamma^{Q+1}) x^2 + \Tr_{q^n/q}(\alpha\gamma^Q+\alpha^Q\gamma)x + \Tr_{q^n/q}(\alpha^{Q+1}).
\]
Since $\gamma^{Q^2-1}=-1$, we have $\gamma^{Q+Q^2}=-\gamma^{Q+1}$ and thus
\[
\Tr_{q^n/q}(\gamma^{Q+1}) = \Tr_{q^n/Q^2}(\Tr_{Q/q}(\Tr_{Q^2/Q}(\gamma^{Q+1}))) = 0.
\]
Likewise,
\[
\Tr_{q^n/q}(\alpha\gamma^Q + \alpha^Q\gamma) =
\Tr_{q^n/q}(\alpha^Q\gamma^{Q^2} + \alpha^Q\gamma) = \Tr_{q^n/q}(0) = 0.
\]
Thus $h_{\alpha}$ induces the same function on $\F_q$ as does the polynomial
$X+\Tr_{q^n/q}(\alpha^{Q+1})$, so that indeed $h_{\alpha}$ permutes $\F_q$ as required.
\end{proof}

\begin{remark} \label{finalrem}
The proof of Theorem~\ref{case9}  shows that every $\alpha \in \F_{q^n}$ satisfies
\[
\Tr_{q^n/q}((\alpha+\gamma x)^{Q+1}) - \Tr_{q^n/q}(\alpha^{Q+1}) =0,
\]
 so that $\gamma$ is a linear translator of $\Tr_{q^n/q}(x^{Q+1})$ (as was claimed on page 3). Theorem~\ref{case9} is a generalization of \cite[Thm.~6]{charpin-kyureg-fq}, which addressed the case of prime $q$.
 \end{remark}



\end{document}